\documentclass{amsart}
\usepackage{amssymb}
\usepackage{amsmath,enumerate}
\usepackage{graphicx}
\usepackage{color}

\theoremstyle{definition}
\newtheorem{Def}{Definition}[section]
\theoremstyle{remark}
\newtheorem{Rem}[Def]{Remark}
\newtheorem{Ex}[Def]{Example}
\newtheorem{Nota}[Def]{Notation}

\theoremstyle{plain}
\newtheorem{Th}[Def]{Theorem}
\newtheorem{Prop}[Def]{Proposition}
\newtheorem{Lem}[Def]{Lemma}
\newtheorem{Cor}[Def]{Corollary}
\newtheorem{Fact}[Def]{Fact}

\newcommand{\Z}{\mathbb{Z}}
\newcommand{\N}{\mathbb{N}}
\newcommand{\C}{\mathbb{C}}

\newcommand{\CI}{\mathcal{I}}
\newcommand{\CM}{\mathcal{M}}

\newcommand{\tp}{{\ }^{t}}
\newcommand{\al}{\alpha }
\newcommand{\be}{\beta }
\newcommand{\ga}{\gamma }

\newcommand{\Ga}{\Gamma }
\newcommand{\De}{\Delta }

\newcommand{\vep}{\varepsilon }

\newcommand{\la}{\lambda }

\newcommand{\pa}{\partial }
\newcommand{\ot}{\otimes }

\newcommand{\one}{\mathbf{1}}
\newcommand{\zero}{\mathbf{0}}

\newcommand{\tM}{\widetilde{M}}
\newcommand{\tH}{\widetilde{H}}
\newcommand{\tF}{\widetilde{F}}
\newcommand{\tG}{\widetilde{G}}

\newcommand{\bfF}{\boldsymbol{F}}
\newcommand{\tbfF}{\widetilde{\boldsymbol{F}}}
\newcommand{\bfw}{\boldsymbol{w}}
\newcommand{\bff}{\boldsymbol{f}}
\newcommand{\bfg}{\boldsymbol{g}}

\newcommand{\ev}{\boldsymbol{e}_{1,\ldots,1}}

\newcommand{\ztwom}{ \{ 0,1 \}^m }
\newcommand{\bfv}{\boldsymbol{v}}
\newcommand{\sol}{Sol_{\dot{x}}}
\newcommand{\bfh}{\boldsymbol{h}}

\title[Irreducibility of the monodromy of $F_C$]
{Irreducibility of the monodromy representation 
of Lauricella's $F_C$
}
\author[Y. Goto]{Yoshiaki Goto}
\address[Goto]{
  General Education,
  Otaru University of Commerce,
  Otaru 047-8501, Japan
}
\email{goto@res.otaru-uc.ac.jp}
\author[K. Matsumoto]{Keiji Matsumoto}
\address[Matsumoto]{
Department of Mathematics,
Hokkaido University,
Sapporo 060-0810, Japan
}
\email{matsu@math.sci.hokudai.ac.jp}

\date{\today}

\begin{document}
\maketitle

\begin{abstract}
Let $E_C$ be the hypergeometric system of differential equations satisfied by 
Lauricella's hypergeometric series $F_C$ of $m$ variables. 
We improve a fundamental system of solutions to $E_C$ expressed
in terms of $F_C$ so that it is valid even in cases where parameters
satisfy some integral conditions.
We show that the monodromy representation of $E_C$ is irreducible under 
our assumption consisting of $2^{m+1}$ conditions for parameters. 
We also show  that the monodromy representation is reducible if 
one of them is not satisfied.

\end{abstract}

\section{Introduction}\label{section-intro}
Lauricella's hypergeometric series $F_C$ of $m$ variables $x_1 ,\ldots ,x_m$ 
with complex parameters $a$, $b$, $c_1$, $\ldots$, $c_m$ is defined by 
\begin{align*}
 F_C (a,b,c ;x ) 
 =\sum_{n_1 ,\ldots ,n_m =0} ^{\infty } 
 \frac{(a,n_1 +\cdots +n_m )(b,n_1 +\cdots +n_m )}
 {(c_1 ,n_1 )\cdots (c_m ,n_m ) n_1 ! \cdots n_m !} x_1 ^{n_1} \cdots x_m ^{n_m} ,
\end{align*}
where $x=(x_1 ,\ldots ,x_m),\ c=(c_1 ,\ldots ,c_m)$, 
$c_1 ,\ldots ,c_m \not\in \{ 0,-1,-2,\ldots \}$, and $(c_1 ,n_1)=\Gamma (c_1+n_1)/\Gamma (c_1)$. 
This series converges in the domain 
$$
D_C =\left\{ (x_1 ,\ldots ,x_m ) \in \C ^m  \ \middle| \ \sum _{k=1} ^{m} \sqrt{|x_k|} <1  \right\} .
$$
It is shown in \cite{HT} that the hypergeometric system $E_C=E_C (a,b,c)$ 
of differential equations satisfied by  $F_C (a,b,c;x)$ is a holonomic system 
of rank $2^m$ with the singular locus 
\begin{align*}
  &S= \Big( \prod_{k=1}^m x_{k} \cdot R(x)=0 \Big) \subset \C^m ,\\
  &R(x_1 ,\ldots ,x_m)=\prod_{\vep_1 ,\ldots ,\vep_m =\pm 1} 
  \Big( 1+\sum_{k=1}^m \vep_k \sqrt{x_k} \Big), 
\end{align*}
and that the system $E_C$ is 
irreducible (in the sense of $D$-modules)
if and only if 
\begin{align}
  \label{irred-1}
  a-\sum_{k=1}^m i_k c_k ,\quad b-\sum_{k=1}^m i_k c_k \not\in \Z ,\qquad 
  \forall I=(i_1,\dots,i_m)\in \ztwom.
\end{align}
It is classically known that 
there are $2^m$ solutions to $E_C (a,b,c)$ 
expressed in terms of $F_C$ with different parameters(see (\ref{series-sol})).  
If parameters satisfy (\ref{irred-1}) and $c_1 ,\ldots ,c_m \not\in \Z$ 
then they form a fundamental system of solutions to $E_C (a,b,c)$ 
in a simply connected domain in $D_C -S$.


Let $X$ be the complement of the singular locus $S$. 
The fundamental group of $X$ is generated by $m+1$ loops 
$\rho_0 ,\ \rho_1 ,\ldots ,\ \rho_m$ (see \S \ref{section-MR}).
In \cite{G-FC-monodromy}, 
we express the circuit transformations 
$\CM_i$ along $\rho_i$ $(i=0,\dots , m)$ 
by using the $2^m$ solutions and the intersection form on 
twisted homology groups associated with 
Euler-type integrals of solutions to $E_C$. 
These expressions are independent of the choice of a basis 
of the twisted homology group. The circuit transformations $\CM_i$ are 
also studied in \cite{M} by the specification of the intersection form 
regarded as indeterminate.

In this paper, we show the following.  
\begin{Th}[Main theorem]
  The monodromy representation 
  $$
  \CM : \pi_1 (X,\dot{x}) \to GL(\sol ) 
  $$
  is irreducible under the assumption (\ref{irred-1}), 
  where $\sol =\sol (a,b,c)$ is the local solution space to $E_C(a,b,c)$ 
  around a point $\dot{x}\in D_C -S$. 
\end{Th}

Note that under the assumption (\ref{irred-1}), for example, 
$c_k$ may be an integer. 
In such a case, the solutions (\ref{series-sol}) expressed by $F_C$ 
do not form a basis of 
the local solution space. 
We give a linear transformation of them 
so that the transformed solutions are valid even in cases where any of 
$c_k$'s are integers.
We construct it inductively on $m$ using tensor products of matrices.

We remark that the irreducibility of the monodromy representation $\CM$ 
is implied from that of the system $E_C$ under the assumption (\ref{irred-1}). 
We prove it explicitly by using properties of the circuit transformations in 
\cite{G-FC-monodromy}, not applying results of $D$-modules.
We here briefly explain our idea of the proof of the main theorem.
It is shown in \cite{G-FC-monodromy} that the $1$-eigenspace $V$ 
of $\CM_0$ is 
$(2^m-1)$-dimensional. Let $f_0\in \sol$ (corresponding to $\ev$ in \S \ref{section-rep-mat}) 
be a non-zero vector in 
its orthogonal complement with respect to the intersection form. 
It is quite easy to give a basis of the whole space $\sol$  
by actions $\CM_1,\dots,\CM_m$ on $f_0$.
Let $W$ be an invariant subspace of $\sol$ under the monodromy 
representation $\CM$. 
If $W\not\subset V$ then we can show $f_0 \in W$, which yields 
that $W$ becomes the whole space $\sol$ by the previous fact.   
Otherwise, we can show that $W$ becomes the zero space 
by the perfectness of the intersection form.  

The irreducibility of $A$-hypergeometric systems is studied in \cite{GKZ}, 
and later in \cite{B} and \cite{SW}. 
Hattori and Takayama \cite{HT} studies 
the irreducibility of the system $E_C(a,b,c)$   
by utilizing results in \cite{B} and \cite{SW} for  
the $A$-hypergeometric systems associated with $F_C$. 
From these results, it seems difficult to know the structure of
an invariant subspace when the system $E_C(a,b,c)$ is reducible.
We show that $E_C(a,b,c)$ is reducible 
if one of the assumption (\ref{irred-1}) is not satisfied, and  
we specify an invariant subspace in $\sol$ under the monodromy 
representation $\CM$ in this case.


\section{Preliminaries}
Except in \S \ref{section-red}, we assume the conditions 
for parameters $a,b,c_1,\dots,c_m$ in (\ref{irred-1}) 
(it is equivalent to (\ref{irred-2}) or (\ref{irred-3}) mentioned below). 

In this section, we collect some facts about Lauricella's $F_C$ 
mentioned in \cite{G-FC}, \cite{G-FC-monodromy}, \cite{HT} and \cite{L}. 
\begin{Nota}
  We put
  \begin{align*}
    \al =\exp (2\pi \sqrt{-1} a),\quad 
    \be =\exp (2\pi \sqrt{-1} b),\quad 
    \ga_k =\exp (2\pi \sqrt{-1} c_k)\ (k=1,\ldots ,m).
  \end{align*}
  We often regard $\al$, $\be$ and $\ga_k$ as indeterminants, 
  and consider the rational function field
  $\C(\al ,\be ,\ga)=\C(\al ,\be ,\ga_1 ,\ldots ,\ga_m)$. 
  For a rational function 
  $g(\al ,\be ,\ga_1 ,\ldots ,\ga_m) \in \C(\al ,\be ,\ga)$, 
  we denote $g(\al ,\be ,\ga_1 ,\ldots ,\ga_m)^{\vee} 
  =g(\al^{-1} ,\be^{-1} ,\ga_1^{-1} ,\ldots ,\ga_m^{-1})$. 
\end{Nota}
Under these notations, the condition (\ref{irred-1}) is equivalent to 
\begin{align}
  \label{irred-2}
  \al-\prod_{k=1}^m \ga_k^{i_k} ,\quad \be-\prod_{k=1}^m \ga_k^{i_k} \neq 0 ,\qquad 
  \forall I=(i_1,\dots,i_m) \in\ztwom.
\end{align}
For example, $\ga_k =1$ or $\al \be -(-1)^{m-1}\prod_{k=1}^m \ga_k =0$ are allowed. 

Note that though in \cite{G-FC-monodromy} the indices $I$ run the subsets of $\{ 1,\dots ,m \}$, 
in this paper we use $\ztwom$ as a set of indices. 
The correspondence is given by 
$$
\{ 1,\dots ,m \} \supset \{ i_1 ,\dots ,i_r \} \longleftrightarrow 
e_{i_1}+\dots+e_{i_r}\in \ztwom ,
$$
where $e_k$ is the $k$-th unit vector of size $m$.
We put $|I|=\sum_{k=1}^m i_k$.

\subsection{System of differential equations}
Let $\pa_k \ (k=1 ,\dots , m)$ be the partial differential operator with respect to $x_k$. 
We set $\theta_k =x_k \pa_k$, $\theta =\sum_{k=1}^m \theta_k$.  
Lauricella's $ F_C (a,b,c;x)$ satisfies differential equations 
$$
\left[ \theta_k (\theta_k+c_k-1)-x_k(\theta +a)(\theta +b)  \right] f(x)=0, 
\quad k=1,\dots , m.
$$ 
The system generated by them is 
called Lauricella's hypergeometric system $E_C (a,b,c)$ of 
differential equations. 
The system $E_{C} (a,b,c)$ is a holonomic system of rank $2^m$ 
with the singular locus $S$. 
It is shown in \cite{HT} that 
the system $E_C(a,b,c)$ is irreducible,
that is, the system $E_{C} (a,b,c)$ defines a maximal ideal in the ring of 
differential operators with rational function coefficients, 
if and only if the parameters $a,b,c_1,\dots,c_m$ satisfy 
(\ref{irred-1}).

For an element $I=(i_1,\dots,i_m)$ of $\ztwom$, we set 
\begin{align}
\label{series-sol}
F_I(x)=\frac{\prod_{k=1}^m\Ga((-1)^{i_k}(1-c_k))}
{\Ga(1-a^I)\Ga(1-b^I)}\cdot \prod_{k=1}^m x_k^{i_k(1-c_k)}
\cdot F_C(a^I,b^I,c^I;x),
\end{align}
where
\begin{align*}
\label{eq:abI}
  &a^I=a+\sum_{k=1}^m i_k(1-c_k),\quad 
  b^I=b+\sum_{k=1}^m i_k(1-c_k), \\
  &c^I=(c_1+2i_1(1-c_1),\dots,c_m+2i_m(1-c_m)).
\end{align*}
Note that the assumption (\ref{irred-1}) is equivalent to
\begin{equation}
\label{irred-3}
a^I,b^I\notin \Z,
\qquad 
  \forall I=(i_1,\dots,i_m)\in \ztwom,
\end{equation}
and that
$$
c_k+2i_k(1-c_k)=\left\{
\begin{matrix} c_k &\textrm{if}& i_k=0,\\
2-c_k &\textrm{if}& i_k=1.\\
\end{matrix}
\right.
$$
\begin{Ex}\label{ex-series-1}
  We give examples for $m=1$ (we put $c_1 =c$, $x_1=x$):
  \begin{align*}
    &F_{0}(x)=\frac{\Ga(1-c)}{\Ga(1-a)\Ga(1-b)}F_C(a,b,c;x)\\
    &=\frac{\Ga(1-c)\Ga(c)}{\Ga(1-a)\Ga(1-b)\Ga(a)\Ga(b)}\sum_{n=0}^\infty
    \dfrac{\Ga(a+n)\Ga(b+n)}{\Ga(c+n)\Ga(1+n)}x^n\\
    &=\frac{\sin(\pi a)\sin(\pi b)}{\pi\sin(\pi c)}\sum_{n=0}^\infty
    \dfrac{\Ga(a+n)\Ga(b+n)}{\Ga(c+n)\Ga(1+n)}x^n,\\
    &F_{1}(x)=\frac{\Ga(c-1)}{\Ga(c-a)\Ga(c-b)}
    x^{1-c}F_C(a+1-c,b+1-c,2-c;x)\\
    &=\frac{\Ga(c\!-\!1)\Ga(1\!-\!(c\!-\!1))}
    {\Ga(c\!-\! a)\Ga(c\!-\! b)\Ga(1\!-\! c\!+\! a)\Ga(1\!-\! c\!+\! b)}
    \sum_{n=0}^\infty \dfrac{\Ga(a\!-\! c\!+\!1\!+\!n)\Ga(b\!-\! c\!+\!1\!+\! n)}
    {\Ga(2\!-\! c\!+\! n)\Ga(1\!+\! n)}
    x^{n+1-c},\\
    &=-\frac{\sin(\pi(a-c))\sin(\pi(b-c))}
    {\pi\sin(\pi c)}
    \sum_{n'=1-c}^\infty \dfrac{\Ga(a\!+\!n')\Ga(b+\! n')}
    {\Ga(c\!+\! n')\Ga(1\!+\! n')}
    x^{n'} ,
  \end{align*}
  where $n'=n+1-c$ and 
  $\sum_{n'=1-c}^\infty$ means the sum of $n'$ running  
  over the set $1-c+\N=\{1-c+n\mid n\in \N\}$. 
\end{Ex}
The functions $\{ F_{I}(x)\}_{I\in \ztwom}$ 
form a basis of the local solution space $\sol =\sol (a,b,c)$ to the system $E_C(a,b,c)$ 
around a point $\dot x$ in $D_C -S$
under conditions (\ref{irred-1}) and
\begin{align}
  \label{c-cond}
  c_1,\dots,c_m\notin \Z. 
\end{align}
We set a (row) vector valued function 
$$\bfF(x)=(\dots, F_{I}(x),\dots),$$
where $I\in \ztwom$ are aligned by the pure lexicographic order as 
$$
(0,\dots,0),\ (1,0,\dots,0), \ (0,1,\dots,0),\  (1,1,\dots,0),\ 
(0,0,1,\dots,0),\ 
\dots, (1,\dots,1).
$$
Note that its entries have factors 
$$1 ,\ x_1^{1-c_1},\ x_2^{1-c_2},\ x_1^{1-c_1}x_2^{1-c_2},\  x_3^{1-c_3},\ 
\dots,\  x_1^{1-c_1}x_2^{1-c_2}\cdots x_m^{1-c_m},$$
respectively.

\subsection{Monodromy representation}
\label{section-MR}
Put $\dot{x}=\left( \frac{1}{2m^2},\ldots ,\frac{1}{2m^2} \right) \in X$. 
For $\rho \in \pi_1 (X,\dot{x})$ and $g \in \sol$, 
let $\rho_* g$ be the analytic continuation of $g$ along $\rho$. 
Since $\rho_* g$ is also a solution to $E_C (a,b,c)$,  
the map $\rho_* :\sol \to \sol ;\ g\mapsto \rho_* g$ is 
a $\C$-linear automorphism which satisfies 
$(\rho \cdot \rho')_* =\rho'_* \circ \rho_*$ for $\rho,\rho' \in \pi_1 (X,\dot{x})$. 
Here, the composition $\rho \cdot \rho'$ of loops $\rho$ and $\rho'$ is 
defined as the loop going first along $\rho$, 
and then along $\rho'$. 
We thus obtain a representation 
$$
\CM :\pi_1 (X,\dot{x}) \to GL(\sol)
$$
of $\pi_1 (X,\dot{x})$, where $GL(V)$ is 
the general linear group on a $\C$-vector space $V$. 
This representation $\CM$ is called the monodromy representation of $E_C(a,b,c)$. 

Let $\rho_0, \rho_1 ,\ldots ,\rho_m$ be loops in $X$ so that 
\begin{itemize}
\item $\rho_0$ turns the hypersurface $(R(x)=0)$ around the point 
  $\left( \frac{1}{m^2},\ldots, \frac{1}{m^2} \right)$, positively, 
\item $\rho_k \ (k=1,\dots , m)$ turns the hyperplane $(x_k=0)$, positively. 
\end{itemize}
For explicit definitions of them, see \cite{G-FC-monodromy}. 
\begin{Fact}[\cite{G-FC-monodromy}]\label{pi1}
  The loops $\rho_0 ,\rho_1 ,\ldots ,\rho_m$ generate 
  the fundamental group $\pi_1 (X,\dot{x})$. 
  Moreover, if $m\geq 2$, then they satisfy the following relations:
  \begin{align*}
    \rho_i \rho_j =\rho_j \rho_i \quad (i,j=1,\dots , m) ,\quad
    (\rho_0 \rho_k)^2 =(\rho_k \rho_0)^2 \quad (k=1,\dots , m).
  \end{align*}
\end{Fact}
In \cite{G-FC-monodromy}, $m+1$ linear maps 
$\CM_i =\CM (\rho_i) \ (i=0,\dots , m)$ are investigated in terms of 
twisted homology groups and the intersection form. 
In this paper, we do not explain them. 
What we need is the following fact.
\begin{Fact}[\cite{G-FC}]\label{solution-cycle}
  Suppose (\ref{irred-1}) and (\ref{c-cond}). 
  \begin{enumerate}[(i)]
  \item By integration, the twisted homology group is isomorphic 
    to the solution space $\sol$. 
  \item We can construct twisted cycles $\{\De_I\}_I$ that correspond to $\{F_I\}_I$. 
  \item The intersection matrix $H=\left( H_{I,I'} \right)_{I,I'}$ 
    with respect to the basis $\{\De_I\}_I$ is 
    diagonal, and its $(I,I)$-entry is
  \begin{align*}
    H_{I,I}
    =\prod_{k=1}^m \frac{(-1)^{i_k} \ga_k^{1-i_k}}{\ga_k-1}
    \cdot \frac{( \al -\prod_{k=1}^m \ga_k^{i_k}) (\be -\prod_{k=1}^m \ga_k^{i_k} ) }
    {( \al -\prod_{k=1}^m  \ga_k ) (\be -1)} .
  \end{align*}  \end{enumerate}
\end{Fact}
By using this fact and properties of the intersection form, 
we can induce the intersection form on $\sol$. 

\begin{Def}
  \label{def-intersection-form}
  We assume (\ref{irred-1}) and (\ref{c-cond}).
  We define a bilinear form 
 (called the intersection form)
  \begin{align*}
    \CI (\cdot ,\cdot ):\sol \times \sol \to \C (\al ,\be ,\ga)
  \end{align*}
  as follows. 
  For any $F(x),G(x) \in \sol$, we express them as 
  linear combinations of the basis $\{ F_I \}_I$: 
  \begin{align*}
    F(x)=\bfF(x)\cdot \bff ,\quad G(x)=\bfF(x)\cdot \bfg,
    \qquad \bff,\bfg \in \C(\al,\be,\ga)^{2^m} ,
  \end{align*}
  and define 
  \begin{align*}
    \CI(F(x),G(x))=\tp \bff \cdot H \cdot \bfg^{\vee} .
  \end{align*}
\end{Def}

\begin{Rem}
\label{rem-zure}
Let $\CI_H$ be the intersection form  on $\sol$ induced from that on 
the twisted homology group by the isomorphism 
in Fact \ref{solution-cycle} (i).  
The intersection form $\CI$ coincides with $\CI_H$ 
modulo a constant multiple which never vanishes 
under the conditions (\ref{irred-1}) and (\ref{c-cond}). 
\end{Rem}

\begin{Cor}
  \label{cor-intersection-form}
  Under the conditions (\ref{irred-1}) and (\ref{c-cond}), 
  the intersection form  $\CI$ is a monodromy invariant form, that is,  
  for any loops $\rho \in \pi_1 (X,\dot{x})$, we have 
  \begin{align*}
    \CI(\CM(\rho)(F(x)),\CM(\rho)(G(x)))=\CI(F(x),G(x)) .
  \end{align*}
  In other words, $H$ satisfies 
  \begin{align*}
    \tp M_{\rho} \cdot H \cdot M_{\rho}^{\vee} =H ,
  \end{align*}
  where $M_{\rho}$ is the representation matrix of $\CM(\rho)$
  with respect to the basis $\{ F_I\}_I$. 
\end{Cor}

Let $M_i$ be the representation matrix of $\CM_i$ $(i=0,\dots , m)$
with respect to the basis $\{ F_I\}_I$. 
We give explicit expressions of them. 
\begin{Fact}[\cite{G-FC-monodromy}]\label{fact:rep-mat-1}
  We assume (\ref{irred-1}), (\ref{c-cond}) and 
  $\la=(-1)^{m-1} \al^{-1} \be^{-1} \prod_{k=1}^m \ga_k \neq 1$. 
  For $k=1,\dots , m$, 
  the representation matrix $M_k$ is diagonal,
  and its $(I,I)$-entry is $\ga_k^{-i_k}$. 
  The representation matrix $M_0$ is expressed as 
$$
  M_0 =E_{2^m} -\frac{1-\la}{\tp \one \cdot H  \cdot\one} 
  \cdot \one \cdot \tp \one \cdot H 
  =E_{2^m} -\frac{(\be-1)(\al -\prod_{k=1}^m \ga_k )}{\al \be} 
  \cdot \one \cdot \tp \one \cdot H ,
$$
  where $E_{2^m}$ is the unit matrix of size $2^m$, 
  $\one$ is the column vector of size $2^m$ with all entries $1$, and 
  $H$ is the intersection matrix given in Fact \ref{solution-cycle}.
\end{Fact}
\begin{Rem}\label{rem-eigen}
  These expressions are obtained from consideration to eigenvectors of each 
  $\CM_i \in GL(\sol )$. 
  \begin{enumerate}[(i)]
  \item $\CM_k$ ($k=1,\ldots,m$); 
    $F_I$ is an eigenvector of eigenvalue $\ga_k^{-1}$ (resp. $1$) 
    if $i_k=1$ (resp. $i_k=0$), where $I=(i_1,\ldots ,i_m)$. 
  \item $\CM_0$; 
    the eigenvalues of $\CM_0$ are 
    $\la$ and $1$. 
    The eigenspace of eigenvalue 
   $\la$ is one-dimensional and spanned by 
    $$
    f_0 = \sum_{I \in \ztwom} F_I , 
    $$
    which corresponds to $\one$ when we take the basis $\{F_I\}_I$. 
    The eigenspace of $\CM_0$ of eigenvalue $1$ is characterized as 
    $\{ g \in \sol \mid \CI (g,f_0 )=0 \}$.
  \item The first expression of $M_0$ is stable under the non-zero scalar 
    multiple to $H$. 
  \end{enumerate}
\end{Rem}

\section{Another basis}
In fact, $\{F_I \}$ does not form a basis of $\sol$ when $c_i$'s are integers. 
In this section, we introduce another basis $\{ \tF_I \}_I$ 
which is a well-defined basis even if $c_i$'s are integers,
and we give the circuit matrices with respect to this basis. 
Note that these do not coincide with 
solutions obtained by integrating the twisted cycles defined in \cite[\S 6]{G-FC-monodromy}

\subsection{Basis of $\sol$}
First, we construct a basis of $\sol$.
\begin{Lem}
\label{lem:scalar-mul}
Let $I=(i_1,i_2\dots,i_m)$ be any element of $\ztwom$ and 
$p=(p_1,p_2,\dots,p_m)$ be any element of $\Z^m$. 
Then the limit function 
$$\lim_{c\to p}(\ga_1-1)(\ga_2-1)\cdots(\ga_m-1)F_{I}(x)$$
is well-defined and not identically zero.
\end{Lem}
\begin{proof}
We have only to note that $F_{I}(x)$ has the factor 
$$\prod_{k=1}^m \Ga(c_k-i_k)\Ga(1-(c_k-i_k))
=\prod_{k=1}^m\frac{(-1)^{i_k}\pi}{\sin(\pi c_k)},
$$
which cancels out $(\ga_1-1)(\ga_2-1)\cdots(\ga_m-1)$.
\end{proof}

\begin{Lem}
\label{lem:wa}
Let $I=(i_1,\dots,\overset{k\textrm{-th}}{0},\dots,i_m)$ and 
$I'=(i_1,\dots,\overset{k\textrm{-th}}{1},\dots,i_m)$ 
be elements of $\ztwom$ and $p_k$ be an integer.
Then the limit function 
$$\lim_{c_k\to p_k}
(\ga_m-1)\cdots(\ga_{k+1}-1)(\ga_{k-1}-1)\cdots(\ga_1-1)
(F_{I}(x)+F_{I'}(x))$$
is well-defined and not identically zero.
\end{Lem}

\begin{proof}
We show the case $m=1$ (we put $p_1=p$). Firstly, we assume $p=1$. 
By Example \ref{ex-series-1}, 
\begin{align*}
F_{0}(x)
%
&=\frac{s_0}{\sin(\pi c)}\sum_{n=0}^\infty
\dfrac{\Ga(a+n)\Ga(b+n)}{\Ga(c+n)\Ga(1+n)}x^n,\\
F_{1}(x)
&=-\frac{s_1(c)}{\sin(\pi c)}
\sum_{n'=1-c}^\infty \dfrac{\Ga(a+n')\Ga(b+ n')}
{\Ga(c+ n')\Ga(1+ n')}x^{n'},  
\end{align*}
where $n'=n+1-c$, $s_0=\dfrac{\sin(\pi a)\sin(\pi b)}{\pi}$, 
$s_1(c)=\dfrac{\sin(\pi(a-c))\sin(\pi(b-c))}{\pi}$. 
Both functions $\sin (\pi c) F_0(x)$ and $-\sin (\pi c) F_1(x)$ 
converge to 
$$
\sum_{n=0}^\infty
s_0\dfrac{\Ga(a+n)\Ga(b+n)}{\Ga(1+n)\Ga(1+n)}x^n
=\sum_{n=0}^\infty A_nx^n$$
as $c\to p=1$.
Apply l'H\^opital's rule to the function 
\begin{align*}
F_0(x)+F_1(x) 
=\dfrac{1}{\sin(\pi c)}\cdot
&\Bigg[\sum_{n=0}^\infty
\big(s_0\dfrac{\Ga(a+ n)\Ga(b+ n)}{\Ga(c+ n)\Ga(1+ n)}- A_n \big)x^n\\
& 
-\sum_{n'=1-c}^\infty 
\Big\{\big(
s_1(c)\dfrac{\Ga(a+n')\Ga(b+ n')}
{\Ga(c+ n')\Ga(1+ n')}- A_n\big)x^{n'}
- A_n(x^n-x^{n'}\big)\Big\}
\Bigg]
\end{align*}
to verify that its limit as $c\to p=1$ exists, where $n=n'+c-1$ 
in the second sum. Note that 
$\lim\limits_{c\to 1} \frac{A_n(x^n-x^{n'})}{\sin(\pi c)}$
yields the factor $\log x$.

Secondly, we assume $p\ge 2$.
In this case, the sum $\sum\limits_{n'=1-c}$ has negative terms 
for $n'=1-p,\dots,-1$ as $c\to p$.  
Since $\lim\limits_{c\to p} \dfrac{1}{\sin(\pi c)}\cdot\dfrac{1}{\Ga(n'+1)}$ 
converges to a non-zero value, these negative terms are well-defined.
Thus $F_0(x)+F_1(x)$ consists of these finite terms and 
the infinite sum considered in the case $p=1$. 

Thirdly, we assume $p\le 0$.
By regarding $n$ as $n+p$ in the sums of $F_0(x)$ and $F_1(x)$, 
we can show that 
$F_0(x)+F_1(x)$ is well-defined and not identically zero as in the previous 
consideration. 

For a general $m$, use a similar argument 
by regrading the variables except $x_k$ as constants. 
Note that the limit function has the factor $\log x_k$  coming from 
$\lim\limits_{n_k\to n'_k} \dfrac{x_k^{n_k}-x_k^{n'_k}}{n_k-n_k'}$.
\end{proof}

We define the tensor product $A\otimes B$ of matrices 
$A$ and $B=(b_{ij})_{\substack{1\le i\le r\\ 1\le j\le s}}$ as 
$$A\otimes B=\begin{pmatrix}
A\, b_{11} & A\, b_{12} &\cdots & A\, b_{1s} \\
A\, b_{21} & A\, b_{22} &\cdots & A\, b_{2s} \\
\vdots  &\vdots   &\ddots & \vdots  \\
A\, b_{r1} & A\, b_{r2} &\cdots & A\, b_{rs} \\
\end{pmatrix}.
$$
We remark that this is different from the usual definition. 
We fix the number $m$ of variables.
We set 
$$
G_k=\begin{pmatrix}
1 & 0 \\
0 & \ga_k^{-1}
\end{pmatrix},\quad 
Q_k=
\begin{pmatrix}
1- \ga_k& 1 \\
0 & 1 
\end{pmatrix},
$$
for $k=1,\dots, m$ and 
\begin{align*}
  P_m=Q_1\otimes Q_2\otimes \cdots \otimes  Q_m.
\end{align*}
By using these notations, 
the matrices $M_1,\dots ,M_m$ given in Fact \ref{fact:rep-mat-1}
is expressed as  
\begin{align*}
  M_k=E_2 \otimes \cdots \otimes E_2 \otimes \underset{k\textrm{-th}}{G_k}
  \otimes E_2 \otimes \cdots \otimes E_2 
  \quad (k=1,\ldots ,m)
\end{align*}
where $E_2$ is the unit matrix of size $2$. 
For example, 
we have 
$$
M_1=G_1=\begin{pmatrix}
1 & 0 \\
0 & \ga_1^{-1}
\end{pmatrix},\quad 
P_1=Q_1=
\begin{pmatrix}
1- \ga_1& 1 \\
0 & 1 
\end{pmatrix}
$$
in the case $m=1$, and 
\begin{align*}
  M_1=\begin{pmatrix}
    1 & 0 & 0 &0 \\
    0 & \ga_1^{-1}& 0 &0\\
    0 & 0 &1 & 0 \\
    0 & 0 &0 &\ga_1^{-1}
  \end{pmatrix},\quad 
  M_2=\begin{pmatrix}
    1 & 0 & 0 &0 \\
    0 & 1 & 0 &0\\
    0 & 0 &\ga_2^{-1} & 0 \\
    0 & 0 &0 &\ga_2^{-1}
  \end{pmatrix},
  \\
  P_2=Q_1\otimes Q_2=\begin{pmatrix}
    (1-\ga_1)(1-\ga_2)& 1-\ga_2 & 1-\ga_1 & 1 \\
    0                & 1-\ga_2 & 0       & 1 \\
    0                & 0       & 1-\ga_1 & 1 \\
    0                & 0       & 0       & 1 
  \end{pmatrix}
\end{align*}
in the case $m=2$. 
Note that 
$$
P_m=\begin{pmatrix}
P_{m-1}(1-\ga_m) & P_{m-1}\\
O & P_{m-1}\\
\end{pmatrix},
$$
where $O$ is the square zero matrix of size $2^{m-1}$.
We have 
\begin{align}
  \label{eq:detPm}
  \det(P_m)=\prod_{k=1}^m(1-\ga_k)^{2^{m-1}},
\end{align}
since $\det(P_1)=1-\ga_1$ and 
$$\det(P_m)=\det((1-\ga_m)P_{m-1})\det(P_{m-1})
=(1-\ga_m)^{2^{m-1}}\det(P_{m-1})^2.$$ 
We use a new basis given by 
$$
\tbfF(x)=(\dots, \tF_I(x),\dots)=\bfF(x)\cdot P_m.
$$
The vector-valued function $\tbfF(x)$ takes the form 
$$
\Big( (1-\ga_1)F_0(x),F_0(x)+F_1(x) \Big)
$$
for $m=1$, and the form 
\begin{align*}
  \Big( &(1-\ga_1)(1-\ga_2)F_{00}(x),
  (1-\ga_2)(F_{00}(x)+F_{10}(x)), \\
  &(1-\ga_1)(F_{00}(x)+F_{01}(x)),
  F_{00}(x)+F_{10}(x)+F_{01}(x)+F_{11}(x) \Big)
\end{align*}
for $m=2$.
\begin{Th}\label{th-sol-basis}
The vector-valued function $\tbfF(x)$ gives a basis of the space $\sol$ of the local solutions  
to $E_C(a,b,c)$ around $\dot x$ even in cases $c_k\in \Z$ $(k=1,\dots,m)$.
\end{Th}
\begin{proof}
The entries of $\tbfF(x)$ consist of 
\begin{align*}
\Big[\prod_{k=1}^m (1-\ga_k)\Big] &F_{0\dots 0}, \\
\Big[\prod_{1\le k\le m}^{k\ne i}(1-\ga_k)\Big] 
&(F_{0\dots 0}+F_{e_i}), \\
\Big[\prod_{1\le k\le m}^{k\ne i,j} (1-\ga_k)\Big] 
&(F_{0\dots 0}+F_{e_i}+F_{e_j}+F_{e_i+e_j}), 
\\
& \vdots \\ 
(1-\ga_k)&\sum_{I\in \ztwom} (1-i_k)F_{I}, \\
&\sum_{I\in \ztwom}F_{I},
\end{align*}
where $e_k$ is the $k$-th unit vector of size $m$ 
and $I=(i_1,\dots,i_m)$. 
The functions in the first and second lines are well-defined by Lemmas 
\ref{lem:scalar-mul} and \ref{lem:wa}.
Since the functions 
\begin{align*}
&\Big[\prod_{1\le k\le m}^{k\ne i}(1-\ga_k)\Big]
\big[
(F_{0\dots 0}+F_{e_i})+(F_{e_j}+F_{e_i+e_j})\big],\\
&\Big[\prod_{1\le k\le m}^{k\ne j}(1-\ga_k)\Big]\big[
(F_{0\dots 0}+F_{e_j})+(F_{e_j}+F_{e_i+e_j})\big]
\end{align*}
are well-defined by Lemma \ref{lem:wa}, the function 
$$
\Big[\prod_{1\le k\le m}^{k\ne i,j}(1-\ga_k)\Big]
(F_{0\dots 0}+F_{e_i}+F_{e_j}+F_{e_i+e_j})$$
is also well-defined. In this way, we can show that
the entries of $\tbfF(x)$ are well-defined 
even in cases $c_k\in \Z$. In the case $c_k\in \Z$, the functions $\tF_I(x)$ 
has the factor $\log x_k$, 
if $i_k=1$. 
This implies that $\tF_I(x)$'s are also linearly independent  
in such a case.
\end{proof}

\subsection{Representation matrices and the intersection matrix}
\label{section-rep-mat}
Next, we consider the representation matrices of $\CM_i$'s and 
the intersection matrix with respect to the new basis $\{\tF_I \}_I$. 

In the below discussion, we often use the following equality 
which is shown by a straightforward calculation:
\begin{align}
  \label{eq:sum}
  &\sum_{J\leq I} (-1)^{|J|}\left( \alpha \beta \prod_{k=1}^m \ga_k^{1-j_k} 
    + \prod_{k=1}^m \ga_k^{1+j_k} \right) \\
  \nonumber
  &=\left( \alpha \beta +(-1)^{|I|}\prod_{k=1}^m \ga_k^{i_k} \right)
  \prod_{k=1}^m \left( \ga_k^{1-i_k}(\ga_k-1)^{i_k} \right) , 
\end{align}
where $I=(i_1,\dots ,i_m)$, $J=(j_1,\dots ,j_m)$ and 
we define a partial order $\leq$ on $\ztwom$ by
$$
J\leq I \Longleftrightarrow j_k \leq i_k, \quad k=1,\dots , m.
$$
\begin{Cor}\label{cor:rep-mat}
  Let $\tM_i$ be the representation matrix of $\CM_i$ $(i=0, \dots ,m)$
  with respect to the basis $\{ \tF_I\}_I$. 
  For $k=1,\dots, m$, we have 
  \begin{align*}
    \tM_k =E_2 \ot \cdots \ot E_2 \ot \underset{k\textrm{-th}}{\tG_k} \ot E_2 \ot \cdots \ot E_2 ,
    \quad 
    \tG_k =\begin{pmatrix}1 & -\ga_k^{-1} \\ 0 & \ga_k^{-1} \end{pmatrix}.    
  \end{align*}
  $\tM_0$ is written as 
  \begin{align*}
    \tM_0 =E_{2^m} -N_0 ,\quad 
    N_0 =\tp (\zero, \dots ,\zero ,\bfv) ,
  \end{align*}
  where $\bfv \in \C^{2^m}$ is a column vector whose $I$-th entry is 
  \begin{align*}
    \left\{ 
      \begin{array}{ll}
        (-1)^m \frac{(\al-1)(\be-1)\prod_{k=1}^m \ga_k}{\al \be}& 
       (I=(0,\dots ,0)) ,\\
        (-1)^{m+|I|} \frac{(\al \be +(-1)^{|I|} \prod_{k=1}^m \ga_k^{i_k})
          \prod_{k=1}^m \ga_k^{1-i_k}}{\al \be}& (I\neq(0,\dots ,0)) .
      \end{array}
    \right.
  \end{align*}
Under the condition (\ref{irred-1}) without assuming  (\ref{c-cond}),
$\tM_0 ,\dots ,\tM_m$ are valid. 
\end{Cor}


\begin{proof}
  By the definition of $\{ \tF_I\}_I$, we have $\tM_i =P_m^{-1} M_i P_m$. 
  The first claim follows from $Q_k^{-1} G_k Q_k =\tG_k$ and 
  the expressions of $M_k$ and $P_k$ as tensor products. 
  
  We show the second claim. 
  Note that 
  since all of the entries of the $2^m$-th column 
  (we also say the $(1,\dots ,1)$-th column)
  of $P_m$ are $1$, we have $P_m \ev =\one$ 
  and hence $\ev$ is an eigenvector of $\tM_0$ of eigenvalue 
  $(-1)^{m-1} \al^{-1} \be^{-1} \prod_k \ga_k$, 
  where $\ev =\tp (0,\dots ,0 ,1)\in \C^{2^m}$. 
  By $\tM_0 =P_m^{-1} M_0 P_m$ and Fact \ref{fact:rep-mat-1}, 
  $\tM_0 -E_{2^m}$ should be
  \begin{align*}
    &\frac{(\be-1)(\al -\prod_{k=1}^m \ga_k )}{\al \be} P_m^{-1} 
    \cdot \one \cdot \tp \one \cdot H P_m
    = \frac{(\be-1)(\al -\prod_{k=1}^m \ga_k )}{\al \be} \ev \tp \one H P_m \\
    &=\frac{(\be-1)(\al -\prod_{k=1}^m \ga_k )}{\al \be} 
    \tp (\zero, \dots ,\zero ,\bfh) \cdot P_m , 
  \end{align*}
  where $\bfh \in \C^{2^m}$ is a column vector whose $I$-th entry is $H_{I,I}$. 
  It is sufficient to show that 
  $$
  \frac{(\be-1)(\al -\prod_{k=1}^m \ga_k )}{\al \be}
  \tp \bfh \cdot P_m 
  =\tp \bfv .
  $$
  The $I$-th entry of the left-hand side is equal to  
  \begin{align*}
    &\frac{(\be-1)(\al -\prod_{k=1}^m \ga_k )}{\al \be}
    \sum_{J\leq I} \left( H_{J,J} \cdot \prod_{k=1}^m (1-\ga_k)^{1-i_k} \right) \\
    &=\frac{(-1)^{m}}{\al \be \prod_{k=1}^m (1-\ga_k)^{i_k}}
    \sum_{J\leq I} \left( (-1)^{|J|}
      \cdot \prod_{k=1}^m  \ga_k^{1-j_k}
      \cdot \Big( \al -\prod_{k=1}^m \ga_k^{j_k} \Big) 
      \Big(\be -\prod_{k=1}^m \ga_k^{j_k} \Big)  \right) .
  \end{align*}
  If $I=(0,\dots ,0)$, then this is the $(0,\dots ,0)$-th entry of $\bfv$. 
  If we assume $I\neq (0,\dots ,0)$, 
  then it equals to 
  \begin{align*}
    &\frac{(-1)^{m}}{\al \be \prod_{k=1}^m (1-\ga_k)^{i_k}}
    \sum_{J\leq I} (-1)^{|J|} 
    \left( \alpha \beta \prod_{k=1}^m \ga_k^{1-j_k} 
      + \prod_{k=1}^m \ga_k^{1+j_k} \right) \\
    &=\frac{(-1)^{m}}{\al \be \prod_{k=1}^m (1-\ga_k)^{i_k}}
    \left( \alpha \beta +(-1)^{|I|}\prod_{k=1}^m \ga_k^{i_k} \right)
    \prod_{k=1}^m \left( \ga_k^{1-i_k}(\ga_k-1)^{i_k} \right) 
  \end{align*}
  by (\ref{eq:sum}), 
  and this coincides with the $I$-th entry of $\bfv$. 
\end{proof}

\begin{Lem}
  \label{lem-indep}
  $2^m$ vectors 
  \begin{align*}
    \left( \prod_{k=1}^m \tM_k^{i_k} \right) \cdot \ev= \tM_1^{i_1} \tM_2^{i_2} \cdots \tM_m^{i_m}\ev 
    \quad (I= ( i_1 ,\ldots i_m ) \in \ztwom )
  \end{align*}
  are linearly independent. 
  In other words, actions $\CM_1,\dots,\CM_m$ on $f_0$ give a basis of the whole space $\sol$.  
\end{Lem}
\begin{proof}
  It is sufficient to show that the $2^m \times 2^m$ matrix
  $$
  (\ev, \tM_1 \ev, \tM_2 \ev , \tM_1 \tM_2 \ev ,\tM_3 \ev ,\ldots , \tM_1 \cdots \tM_m \ev )
  $$
  is invertible. 
  We calculate its determinant. 
  Because of $\tM_k=P_m^{-1}M_k P_m$ and $\ev=P_m^{-1}\one$, this matrix equals to 
  \begin{align}
    \label{eq-basis-matrix}
    P_m^{-1} \cdot (\one, M_1 \one, M_2 \one, M_1 M_2 \one ,M_3 \one,\ldots ,M_1 \cdots M_m \one ) .
  \end{align}
  By the alignment of the indices set, the right side of this product is 
  \begin{align*}
    \begin{pmatrix}
      1 & 1 \\ 1 & \ga_1^{-1}
    \end{pmatrix} 
    \ot
    \begin{pmatrix}
      1 & 1 \\ 1 & \ga_2^{-1}
    \end{pmatrix} 
    \ot \cdots \ot
    \begin{pmatrix}
      1 & 1 \\ 1 & \ga_m^{-1} 
    \end{pmatrix} ,
  \end{align*}
  and its determinant is $\prod_{k=1}^m (\ga_k^{-1}-1)^{2^{m-1}}$. 
  By (\ref{eq:detPm}), the determinant of (\ref{eq-basis-matrix}) is equal to 
  $\prod_{k=1}^m \ga_k^{-2^{m-1}}$, 
  which is not zero. 
\end{proof}

\begin{Prop}
  \label{lem-intersection}
  Let $\tH=\tp P_m H P_m^{\vee}$, which represents  
  the intersection form $\CI$ with respect to the basis $\{ \tF_I\}_I$. 
  Then $\tH$ 
  is well-defined, and its determinant is 
  \begin{align*}
    \det(\tH)=\frac{1}{(\al -\prod _k  \ga_k )^{2^m}(\be -1)^{2^m}}
    \cdot \prod_{I\in \ztwom}\left(\al-\prod_{k=1}^m \ga_k^{i_k}\right)
    \left(\be-\prod_{k=1}^m \ga_k^{i_k}\right).
  \end{align*}
  In particular, $\tH$ is non-degenerate 
  even in cases $c_k\in \Z$ $(k=1,\dots,m)$.
\end{Prop}
\begin{proof}
  First, we show the well-definedness. 
  For $I=(i_1,\dots ,i_m)$, $I'=(i'_1,\dots ,i'_m)$, we put 
  $$
  I\cdot I'=(i_1 i'_1 ,\dots ,i_m i'_m )\in \ztwom .
  $$
  Since $H$ is diagonal, the $(I,I')$-entry of $\tH$ is  
  \begin{align*}
    &\tH_{I,I'} \\
    &=\prod_{k=1}^m (1-\ga_k)^{1-i_k} (1-\ga_k^{-1})^{1-i'_k}
    \sum_{J\leq I, J\leq I'} \left( \prod_{k=1}^m \frac{(-1)^{j_k} \ga_k^{1-j_k}}{\ga_k-1}
      \cdot \frac{( \al -\prod_{k=1}^m \ga_k^{j_k}) (\be -\prod_{k=1}^m \ga_k^{j_k} ) }
      {( \al -\prod_{k=1}^m  \ga_k ) (\be -1)} \right) \\
    &=\frac{\prod_{k=1}^m (-1)^{i'_k} \ga_k^{i'_k-1}(1-\ga_k)^{1-i_k-i'_k}}
    {( \al -\prod_{k=1}^m  \ga_k ) (\be -1)} 
    \sum_{J\leq I \cdot I'}  \left( (-1)^{|J|} \prod_{k=1}^m \ga_k^{1-j_k}
      \cdot \Big( \al -\prod_{k=1}^m \ga_k^{j_k} \Big) \Big(\be -\prod_{k=1}^m \ga_k^{j_k} \Big) \right) .
  \end{align*}
  If $I\cdot I'=(0,\dots ,0)$, then $1-i_k-i'_k \geq 0\ (k=1,\dots ,m)$, and hence
  \begin{align*}
    \tH_{I,I'}
    =\prod_{k=1}^m (-\ga_k)^{i'_k}(1-\ga_k)^{1-i_k-i'_k} 
    \cdot \frac{\al -1}{\al -\prod_{k=1}^m  \ga_k}  
  \end{align*}
  is well-defined. 
  If $I\cdot I' \neq (0,\dots ,0)$, the same calculation as the proof of 
  Corollary \ref{cor:rep-mat} shows 
  \begin{align*}
    &\tH_{I,I'}\\
    &=\frac{\prod_{k=1}^m (-1)^{i'_k} \ga_k^{i'_k-1}(1-\ga_k)^{1-i_k-i'_k}}
    {( \al -\prod_{k=1}^m  \ga_k ) (\be -1)} 
    \left( \alpha \beta +(-1)^{|I\cdot I'|}\prod_{k=1}^m \ga_k^{i_k i'_k} \right)
    \prod_{k=1}^m \left( \ga_k^{1-i_k i'_k}(\ga_k-1)^{i_k i'_k} \right) \\
    &=\frac{\alpha \beta +(-1)^{|I\cdot I'|}\prod_{k=1}^m \ga_k^{i_k i'_k}}
    {( \al -\prod_{k=1}^m  \ga_k ) (\be -1)}
    \cdot \prod_{k=1}^m (-\ga_k)^{i'_k(1-i_k)} (1-\ga_k)^{(1-i_k)(1-i'_k)}, 
  \end{align*}
  and we can see that 
  its denominator does not vanish.

  Next, we evaluate $\det (\tH)$. 
  Straightforward calculation and (\ref{eq:detPm}) show
  \begin{align*}
    \det (H) &= (-1)^{m 2^{m-1}} \cdot 
    \frac{\prod_{k=1}^m \ga_k^{2^{m-1}} 
      \cdot \prod_I( \al -\prod_{k=1}^m \ga_i^{i_k}) (\be -\prod_{k=1}^m \ga_i^{i_k} ) }
    {\prod_{k=1}^m (\ga_k -1)^{2^m} \cdot ( \al -\prod_{k=1}^m \ga_k )^{2^m} (\be -1)^{2^m}} , \\
    \det (P_m^{\vee})
    &=\prod_{k=1}^m \left( \ga_k^{-2^{m-1}}(\ga_k-1)^{2^{m-1}} \right) .
  \end{align*}
  We thus have 
  \begin{align*}
    \det (\tH )=\frac{1}{(\al -\prod_{k=1}^m \ga_k )^{2^m}(\be -1)^{2^m}}
    \cdot \prod_{I\in \ztwom}\left(\al-\prod_{k=1}^m \ga_k^{i_k}\right)
    \left(\be-\prod_{k=1}^m \ga_k^{i_k}\right) ,
  \end{align*}
  and it is not zero under the condition (\ref{irred-2}). 
\end{proof}
By this proposition, we can relax the condition to define 
the intersection from on $\sol$. 
\begin{Cor}
  By using $\{ \tF_I\}_I$ and $\tH$, the intersection form 
  $\CI :\sol \times \sol \to \C (\al ,\be ,\ga)$ 
  in Definition \ref{def-intersection-form}
  can be extended even in cases $c_k\in \Z$ $(k=1,\dots,m)$.
\end{Cor}

\begin{Lem}
  \label{lem-eigenspace}
  The eigenspace of $\tM_0$ with eigenvalue $1$ is expressed as
  \begin{align*}
    \ker N_0= \ker \tp \bfv .
  \end{align*}
\end{Lem}
\begin{proof}
  This is obvious because of the expression 
  \begin{align*}
    \tM_0 =E_{2^m}-N_0=E_{2^m}- \tp (\zero,\ldots ,\zero ,\bfv ) .
  \end{align*}
\end{proof}

\begin{Rem}
  If the eigenvalue $(-1)^{m-1}\al^{-1}\be^{-1}\prod_{k=1}^m\ga_k$ of $\tM_0$ coincides with $1$, 
  then the $(1,\dots ,1)$-th entry of $\bfv$ is zero and 
  $\ev$ also belongs to this eigenspace $\ker N_0$. 
\end{Rem}

\begin{Lem}
  \label{lem-orthogonal}
  $$
  \ker N_0 
  =\{ \bfw\in \C^{2^m} \mid \tp \bfw \tH \ev =0\}.
  $$ 
\end{Lem}
\begin{proof}
  This is also obvious because of the orthogonality of the eigenspaces (Remark \ref{rem-eigen} (ii))
  and the definition of $\tH$. 
\end{proof}

\section{Irreducibility}
We restate the main theorem and give its proof. 
\begin{Th}
  The monodromy representation 
  $$
  \CM : \pi_1 (X,\dot{x}) \to GL(\sol )
  $$
  is irreducible under the condition (\ref{irred-1}). 
\end{Th}
\begin{proof}
  By Theorem \ref{th-sol-basis}, 
  it is sufficient to consider the matrix representation by $\tM_i$ under 
  the isomorphism $\sol \simeq \C^{2^m}$. 
  Let $W\subset \C^{2^m}$ be an invariant subspace. 
  \begin{enumerate}
    \item[(i)] First, we suppose $W \not\subset \ker N_0$.  
      We take $\bfw\in W$ such that $N_0 \bfw \neq \zero$. 
      By the definition of $N_0$, the image of $N_0$ is spanned by $\ev$. 
      Thus $N_0 \bfw \neq \zero$ implies that 
      there exists $\mu \neq 0$ such that $N_0 \bfw =\mu \ev$. 
      We obtain 
      \begin{align*}
        \ev=\frac{1}{\mu}N_0 \bfw =\frac{1}{\mu} (\bfw-\tM_0 \bfw) \in W.
      \end{align*}
    By Lemma \ref{lem-indep}, the $2^m$ vectors
    \begin{align*}
      \left( \prod_{k=1}^m \tM_k^{i_k} \right) \cdot \ev\in W 
      \qquad (I= ( i_1 ,\ldots i_m ) \in \ztwom )      
    \end{align*}
    are linearly independent. 
    This implies $W=\C^{2^m}$. 
    \item[(ii)] Next, we suppose $W \subset \ker N_0$. 
      We fix an arbitrary $\bfw\in W$. 
      Since $W$ is an invariant subspace, we have
      \begin{align*}
        (\tM_1^{i_1})^{-1}(\tM_2^{i_2})^{-1}\cdots(\tM_m^{i_m})^{-1}\bfw \in W \subset \ker N_0 
      \end{align*}
      for any $I=(i_1 ,\ldots i_m ) \in \ztwom$.
      By the monodromy invariance $\tp \tM_i \tH \tM_i^{\vee}=\tH$ of 
      the intersection matrix $\tH$, 
      commutativity between $\tM_1, \tM_2,\ldots ,\tM_m$, and
      Lemma \ref{lem-orthogonal}, we obtain  
      \begin{align*}
        &\tp \bfw \tH (\tM_1^{i_1} \tM_2^{i_2} \cdots \tM_m^{i_m} \ev)^{\vee} \\
        &=\tp \bfw \tp ((\tM_1^{i_1})^{-1}(\tM_2^{i_2})^{-1}\cdots(\tM_m^{i_m})^{-1}) \tH \ev^{\vee} \\
        &=\tp ((\tM_1^{i_1})^{-1}(\tM_2^{i_2})^{-1}\cdots(\tM_m^{i_m})^{-1} \bfw) \tH \ev^{\vee} =0 .
      \end{align*}
      The linear independence of $\{\tM_1^{i_1} \tM_2^{i_2} \cdots \tM_m^{i_m}\ev\}_I$ and  
      $\det (\tH ) \neq 0$ (Proposition \ref{lem-intersection}) means that $\bfw=0$. 
      We thus have $W=0$. 
    \end{enumerate}
  Therefore, the invariant subspaces should be the trivial ones. 
\end{proof}

\section{Reducibility}\label{section-red}
Recall that our irreducibility assumption (\ref{irred-1}) consists of 
$2^{m+1}$ conditions for parameters. 
In this section, we show that if one of them is not satisfied
then the monodromy representation $\CM$ of $E_C(a,b,c)$ is reducible.
More precisely, we have the following theorem. 

\begin{Th}
\label{th-reducible}
Suppose that there exists $I=(i_1,\dots,i_m)\in \ztwom$ such that 
$a^I\in \Z$, $b^I\notin \Z$ or  $a^I\notin \Z$, $b^I\in \Z$. 
If $a^{I'},b^{I'}\notin \Z$ for any $I'\in \ztwom$ different from $I$,  
then the monodromy representation $\CM$ of $E_C(a,b,c)$ is reducible, that is, 
there exists a non-trivial subspace in $\sol$ invariant under $\CM$. 
\end{Th}

\begin{proof}
We fix $I=(i_1,\dots,i_m)\in \ztwom$ and 
assume that $a^I\in \Z,\ b^I\notin \Z.$
If there exists $j$ such that $c_j\in \Z$ then 
$$
\left\{
\begin{matrix}
  a^{I+e_j}=a^I+ (1-c_j)\in \Z& \textrm{if} & i_j=0,\\
  a^{I-e_j}=a^I- (1-c_j)\in \Z& \textrm{if} & i_j=1,
\end{matrix}
\right.
$$
which contradicts to the assumption.
Thus we have $c_1,\dots,c_m\notin \Z$ and 
the solutions $F_{I'}(x)$ in (\ref{series-sol}) for $I'(\ne I)$ are valid. 
Note that these $2^m-1$ solutions are linearly independent.

Hereafter, we regard $a^I$ as an indeterminant, 
and consider two cases:
\begin{enumerate}[(i)]
\item $a^I$ approaches to a non-positive integer $-L$ 
  ($L\in \{ 0,1,2,\dots \}$);
\item $a^I$ approaches to a positive integer $L'$ 
  ($L'\in \{ 1,2,3,\dots \}$). 
\end{enumerate}
To prove the reducibility, we find a non-trivial
invariant subspace in each case.
\begin{enumerate}[(i)]
\item 
Note that the solution $F_{I}(x)$ in (\ref{series-sol}) for this $I$
is expressed as a non-zero constant multiple of 
$$
\lim_{a^I\to -L} \sin(\pi a^I)\sum_{(n_1',\dots,n_m')}
\frac{\Ga(a+n'_1+\cdots+ n'_m)\Ga(b+n'_1+\cdots+ n'_m)}
{\Ga(c_1\!+\!n_1')\cdots \Ga(c_m\!+\!n_m')\Ga(1\!+\!n_1')\cdots 
\Ga(1\!+\!n_m')}x_1^{n_1'}\cdots x_m^{n_m'},
$$
where $n_k'$ $(1\le k\le m)$ runs over the set 
$$
\left\{
\begin{array}{lcc}
\N=\{n_k\mid n_k\in \N\}=\{0,1,2,\dots\} & \textrm{if} & k\notin I,\\
1-c_k+\N=\{1-c_k+n_k\mid n_k\in \N\} & \textrm{if} & k\in I.
\end{array}
\right.
$$
Since 
$$a+n'_1+\cdots+ n'_m=a^I+n_1+\cdots+n_m=-L+n_1+\cdots+n_m\le 0
$$
for $n_1+\cdots +n_m\le L$ when $a^I\to -L$, 
some finite terms of $\Ga(a+n'_1+\cdots+ n'_m)$ diverge.  
However, the poles of the Gamma function are simple, 
these poles are canceled by the limit of $\sin(\pi a^I)$ 
as $a^I\to -L$.  
Hence the solution $F_I(x)$ is reduced to the sum of finite terms 
by the limit $a^I\to -L$.  Since $F_I(x)$ is a polynomial times 
$\prod_{k=1}^m x_k^{i_k(1-c_k)}$, 
the $1$-dimensional span of $F_I(x)$ in $Sol_{\dot x}$ 
is invariant under $\CM$. Therefore the monodromy representation $\CM$ is 
reducible in this case. 
Note that the representation matrices $M_i$ $(i=0,1,\dots,m)$ in 
Fact \ref{fact:rep-mat-1} are 
valid under the limit and the $I$-th column of $M_0$ is 
the $I$-th unit column vector of size $2^m$. We can also see that 
the $1$-dimensional span of $F_I(x)$ is invariant under $\CM$ 
by these representation matrices.
\item
If the parameter $a^I$ goes to a positive integer $L'$, 
then the solution $F_{I}(x)$ in (\ref{series-sol}) for this $I$
reduces to the identically zero. 
Thus we use the fundamental system 
$(\dots, F_I'(x),\dots)=(\dots, F_I,\dots)H^{-1}$, 
where the diagonal matrix $H$ is 
given in Fact \ref{solution-cycle}. 
By the explicit form of $H$, we can easily see that 
$F_{I'}'(x)$ for $I'\ne I$ are valid under the limit $a^I \to L'$.
The limit $\lim\limits_{a^I \to L'} F_I'(x)$ 
is a non-zero constant multiple of 
$$
\lim_{a^I\to L'} \sum_{(n_1',\dots,n_m')}
\frac{\Ga(a+n'_1+\cdots+ n'_m)\Ga(b+n'_1+\cdots+ n'_m)}
{\Ga(c_1+n_1')\cdots \Ga(c_m+n_m')\Ga(1+n_1')\cdots \Ga(1+n_m')}
x_1^{n_1'}\cdots x_m^{n_m'},
$$
where $n_k'$ $(1\le k\le m)$ runs over the same set as (i).
Since each term of this series converges as $a^I\to L'$, 
this limit is a solution to $E_C(a,b,c)$ with 
the factor $\prod_{k=1}^m x_k^{i_k(1-c_k)}$.
Thus the fundamental system $(\dots, F_I'(x),\dots)$ is valid 
under the limit $a^I \to L'$.
By this change of fundamental systems, the representation matrices 
$M_i$ $(i=0,1,\dots,m)$ are transformed into 
$$M'_i=H M_i H^{-1}.
$$
For $k=1,\dots,m$, 
since $M_k$ and $H$ are diagonal, we have 
$M'_k=H M_k H^{-1}=M_k=\tp M_k$. 
By Fact \ref{fact:rep-mat-1}, $M_0'$ is given as
\begin{align*}
M'_0=HM_0H^{-1}&=
H\Big(E_{2^m} -\frac{(\be-1)(\al -\prod_{k=1}^m \ga_k )}{\al \be} 
  \cdot \one \cdot \tp \one \cdot H\Big)H^{-1}\\
&=E_{2^m} -\frac{(\be-1)(\al -\prod_{k=1}^m \ga_k )}{\al \be} 
  \cdot H\cdot \one \cdot \tp \one = \tp M_0.
\end{align*}
These representation matrices are valid under the limit $a^I\to L'$.
By this limit, the $I$-th row of $M'_0$ is the $I$-th 
unit row vector of size $2^m$. 
Hence the $(2^m-1)$-dimensional space spanned by $F'_{I'}$ $(I'\ne I)$ 
is invariant under $\CM$. 
\end{enumerate}
Therefore, we obtain non-trivial invariant subspaces, and 
complete the proof.  
\end{proof}

\begin{Rem}
Even in the case of $m=1$, we need detailed case analysis
to give a fundamental system of solutions to $E_C(a,b,c)$ 
in terms of the series (\ref{series-sol})
without the condition (\ref{irred-1}), refer to \cite{KS} and \cite{MS}.
\end{Rem}


\begin{thebibliography}{99}
%
%
\bibitem{B} 
F. Beukers, 
\emph{Irreducibility of A-hypergeometric systems}, 
{Indag. Math. (N.S.)}, \textbf{21} (2011), 30--39.
%
%
\bibitem{GKZ}
I.M. Gel'fand,  M.M. Kapranov and A.V. Zelevinsky,  
\emph{Generalized Euler integrals and $A$-hypergeometric functions}, 
{Adv. Math.}, \textbf{84} (1990), 255--271.
%
\bibitem{G-FC}
Y. Goto,
\emph{Twisted cycles and twisted period relations
for Lauricella's hypergeometric function $F_C$}, 
{Internat. J. Math.}, \textbf{24} (2013), 1350094 19pp.
%
\bibitem{G-FC-monodromy}
Y. Goto,
\emph{The monodromy representation of Lauricella's hypergeometric function $F_C$}, 
{Ann. Sc. Norm. Super. Pisa Cl. Sci. (5)}, \textbf{XVI} (2016), 1409--1445. 
%
%
%
\bibitem{HT}
R. Hattori and N. Takayama, 
\emph{The singular locus of Lauricella's $F_C$}, 
{J. Math. Soc. Japan}, \textbf{66} (2014), 981--995.
%
%
%
\bibitem{KS}
T. Kimura and K. Shima, 
\emph{A note on the monodromy of the hypergeometric differential equation},
 Japan. J. Math. (N.S.) 
\textbf{17} (1991), 
137--163.
%
%
\bibitem{L}
G. Lauricella, 
\emph{Sulle funzioni ipergeometriche a pi\`u variabili}, 
{Rend. Circ. Mat. Palermo}, 
\textbf{7} (1893), 111--158.
%
\bibitem{M} K. Matsumoto, 
\emph{
Monodromy representations of hypergeometric systems 
with respect to fundamental series solutions}, 
to appear in {Tohoku Math. J.}, 
arXiv:1502.01826.
%
\bibitem{MS}
K. Mimachi and T. Sasaki,  
\emph{Monodromy representations associated with the Gauss hypergeometric 
function using integrals of a multivalued function}, 
{Kyushu J. Math.} \textbf{66} (2012), 
35--60. 
%
%
%
\bibitem{SW}
M. Schulze and U. Walther, 
\emph{Resonance equals reducibility for A-hypergeometric systems}, 
{Algebra Number Theory}, \textbf{6} (2012), 527--537.
%
%

\end{thebibliography}
\end{document}